\documentclass[10pt]{article}
\usepackage{t1enc}
\usepackage[latin1]{inputenc}
\usepackage[english]{babel}
\usepackage{amsmath,amsthm}
\usepackage{amsfonts}
\usepackage{latexsym}
\usepackage[dvips]{graphicx}
\usepackage{graphicx}
\usepackage[natural]{xcolor}
\usepackage{algorithm}
\usepackage{algorithmic}
\usepackage{multirow}
\usepackage{xcolor}
\usepackage{enumerate}
\usepackage{listings}
\newtheorem{theorem}{Theorem}
\newtheorem{lemma}[theorem]{Lemma}
\newtheorem{claim}[theorem]{Claim}
\newtheorem{proposition}[theorem]{Proposition}

\newtheorem{conjecture}[theorem]{Conjecture}
\newtheorem*{1-2-3 conjecture}{1-2-3 Conjecture}
\newtheorem*{elzahar conjecture}{El-Zah\'ar's Conjecture}
\newtheorem*{moon and moser thm}{Moon and Moser's Theorem}

\newtheorem{case}{Case}

\title{Rainbow Pancyclicity in Graph Systems}
\author{Yangyang Cheng$^{a,}$\unskip\thanks{\emph{E-mail address:} mathsoul@mail.sdu.edu.cn }, Guanghui Wang$^{a,}$\unskip\thanks{Corresponding author. \emph{E-mail address:} ghwang@sdu.edu.cn}, Yi Zhao$^{b,}$\unskip\thanks{\emph{E-mail address:} yzhao6@gsu.edu}\\ \\[.5em]
{\small $^a$School of Mathematics,}\\
{\small Shandong University, 250100, Jinan, Shandong, P. R. China}\\
{\small $^b$Department of Mathematics and Statistics,}\\
{\small Georgia State University, Atlanta, GA 30303, USA}\\}
\date{}
\begin{document}
\baselineskip 0.65cm
\maketitle
\begin{abstract}
 Let $G_1,...,G_n$ be graphs on the same vertex set of size $n$, each graph with minimum degree $\delta(G_i)\ge n/2$. A recent conjecture of Aharoni asserts that there exists a rainbow Hamiltonian cycle i.e. a cycle with edge set $\{e_1,...,e_n\}$ such that $e_i\in E(G_i)$ for $1\leq i \leq n$. This can be viewed as a rainbow version of the well-known Dirac theorem.
In this paper, we prove this conjecture asymptotically by showing that for every $\varepsilon>0$, there exists an integer $N>0$, such that when $n>N$ for any graphs $G_1,...,G_n$ on the same vertex set of size $n$ with $\delta(G_i)\ge (\frac{1}{2}+\varepsilon)n$, there exists a rainbow Hamiltonian cycle. Our main tool is the absorption technique. Additionally, we prove that with $\delta(G_i)\geq \frac{n+1}{2}$ for each $i$, one can find rainbow cycles of length $3,...,n-1$.
\end{abstract}
\bigskip
\noindent {\bf Keywords:} Dirac theorem; rainbow Hamiltonian cycle; absorption technique; pancyclicity
\section{Introduction}
Let $G_1,...,G_t$ be $t$ graphs on the same vertex set $V$ of size $n$ where $t$ is a positive integer. We denote the edge set of $G_i$ by $E(G_i)$ and assume that each edge in $E(G_i)$ is coloured by $i$ for $1\leq i \leq t$. Let $S$ be an edge set that is a subset of $\cup_{i=1}^t E(G_i)$ and we say $S$ is \emph{rainbow} if any pair of edges in $S$ have distinct colours. Rainbow Hamiltonian cycles have been studied by many authors. An edge-coloured graph $G$ is \emph{$k$-bounded} if no colour appears more than $k$ times. Erd\H{o}s, Ne\v{s}et\v{r}il and R\"{o}dl \cite{erdos} studied the problem for which $k$ any $k$-bounded $K_n$ contains a rainbow Hamiltonian cycle and they showed that $k$ could be any constant. Hahn and Thomassen \cite{Hahn} demonstrated that $k$ could grow as fast as $n^{\frac{1}{3}}$ and conjectured that the growth of $k$ could in fact be linear. This was confirmed by Albert, Frieze and Reed \cite{Albert}. A recent result from  Coulson and Perarnau \cite{Coulson} further strengthened this by replacing the complete graph with any \emph{Dirac graph}. More precisely, they proved that there exists $\mu>0$ and positive integer $n_0$ such that if $n\geq n_0$ and $G$ is a $\mu n$-bounded edge-coloured graph on $n$ vertices with minimum degree $\delta(G)\geq \frac{n}{2}$, then $G$ contains a rainbow Hamiltonian cycle.
For rainbow Hamiltonian cycles in graph systems, Aharoni et al. \cite{Aharoni} recently gave the following elegant conjecture, which is a natural generalization of Dirac's theorem to the case of graph systems:

\begin{conjecture}\label{rhc}
Given graphs $G_1,...,G_n$ on the same vertex set of size $n$, if each graph has minimum degree at least $\frac{n}{2}$, then there exists a rainbow Hamiltonian cycle.
\end{conjecture}

There have been several papers studying other rainbow structures in graph systems. For example, a well-known conjecture of Aharoni and Berger \cite{Aharoni2} asserts that if $M_1, \dots, M_n$ are $n$ matchings of size at least $n+1$ on the same vertex set $V=X\cup Y$ where $X$ and $Y$ are disjoint and all edges of $M_i$ are between $X$ and $Y$, then there exists a rainbow matching of size $n$.
This conjecture generalizes the famous Brualdi-Stein Conjecture, which asserts that every $n\times n$
Latin square has a partial transversal of size $n-1$.
The Aharoni-Berger Conjecture has been confirmed asymptotically by Pokrovskiy \cite{Pok}. For more details about this topic, see \cite{Wanless}.

In this paper, we prove an asymptotic version of Conjecture \ref{rhc}:

\begin{theorem}\label{main}
For every $\varepsilon>0$, there exists an integer $N>0$, such that when $n>N$ for any graphs $G_1,...,G_n$ on the same vertex set of size $n$, each graph with minimum degree $\delta(G_i)\ge (\frac{1}{2}+\varepsilon)n$, there exists a rainbow Hamiltonian cycle.
\end{theorem}

After we submitted this paper, Joos and Kim in \cite{Kim} proved Conjecture \ref{rhc} using a method different from ours. Nevertheless, we believe that our approach is of independent interest and could be applied to attack similar problems in hypergraphs.

Furthermore, we show that given $n$ graphs $G_1, \ldots, G_n$ with $\delta(G_i)\ge \frac{n+1}{2}$ for $1\leq i \leq n$, we can find rainbow cycles with all possible lengths except a Hamiltonian one:

\begin{theorem} \label{thm1}
Given graphs $G_1,...,G_n$ on the same vertex set of size $n$, each graph with minimum degree $\delta(G_i)\ge \frac{n+1}{2}$, there exist rainbow cycles of length $3,4,\dots,n-1$.
\end{theorem}

Combining Theorem \ref{thm1} with the result of Joos and Kim, we derive that any $G_1,...,G_n$ satisfying the assumption of Theorem \ref{thm1} indeed contain rainbow cycles of all possible lengths $3,...,n$. The lower bound of Theorem \ref{thm1} is tight because one can take $n$ copies of $K_{\frac{n}{2}, \frac{n}{2}}$ where $n$ is even and there does not exist any odd rainbow cycle in such a system.

The main tool behind the proof of Theorem \ref{main} is the absorbing method that was introduced by R\"odl, Ruci\'nski and Szemer\'edi \cite{rodl}. Here we apply a rainbow version of the approach of Lo \cite{lo} by constructing a short rainbow cycle $C$ such that for any rainbow path $P=v_1\cdots v_p$ disjoint from $C$ and a new colour $s$ where the colour set of $P$ is also disjoint from that of $C$, we can absorb $P$ into $C$. In other words, we replace some edge $u_iu_{i+1}$ of $C$ by a path $u_iPu_{i+1}$, where $u_iu_1$ is coloured with $s$ and $v_pu_{i+1}$ is coloured with the colour of $u_iu_{i+1}$ in $C$. Finally, we find a rainbow Hamiltonian path $P$ on $V(G)\backslash V(C)$ and absorb $P$ into $C$ by the property of $C$ and thus obtain a rainbow Hamiltonian cycle.

\section{Preliminaries and Notation}

Let $G_1,...,G_n$ be $n$ graphs on the same vertex set $V$ where $|V|=n$. Let $\delta(G_i)$ be the minimum degree of each $G_i$ for $1\leq i \leq n$. By our assumption, we identify this graph system with an edge-coloured multigraph $G$ where $E(G)$ is the disjoint union of $E(G_i)$ for $i\in [n]$ and each edge in $E(G_i)$ is coloured by $i$. For any subgraph $H$ of $G$, let $\text{Col}(H)$ be the set of colours used by the edges of $H$. For every vertex $v\in V(G)$ and any colour $c\in [n]$, let $N_c(v)$ be the set of neighbours of $v$ that are adjacent to $v$ by an edge coloured by $c$. Let $S$ be any subset of $V$, we denote $N_c(v)\cap S$ by $N_c(v,S)$ and $|N_c(v)\cap S|$ by $d_c(v,S)$. For each pair of vertices $v_1,v_2 \in V(G)$, let $\text{Col}(v_1,v_2)$ be the set of colours used for the edges between $v_1$ and $v_2$ ($\text{Col}(v_1,v_2)$ is empty if there are no edges between $v_1$ and $v_2$).
We will use the following version of Chernoff's bound \cite{Jason}.

\begin{lemma}\label{chernoff}
Let $X$ be a binomially distributed random variable and $0<\varepsilon <\frac{3}{2}$, then
$$P(|X-E(X)|\geq \varepsilon E(X))\leq 2e^{-\frac{\varepsilon^2}{3}E(X)}.$$
\end{lemma}

We first prove the following useful lemma:
\begin{lemma}\label{lemma1}
Let $P=v_1 \cdots v_p$ be a rainbow path and let $c, c^{\prime}$ be two colours not used on $P$. If $d_c(v_1,V(P)) + d_{c^{\prime}}(v_p,V (P))\geq p$, then there is a rainbow cycle of length $p$.
\end{lemma}
\begin{proof}
If $\{c,c^{\prime}\}\cap \text{Col}(v_1,v_p)\neq \emptyset$, then $C=v_1\cdots v_pv_1$ is a rainbow cycle by choosing the colour of $v_1v_p$ to be an element in $\{c,c^{\prime}\}\cap \text{Col}(v_1,v_p)$. So we assume that $\{c,c^{\prime}\}\cap \text{Col}(v_1,v_p)=\emptyset$. Suppose that there exists no rainbow cycle of length $p$. For each vertex $v_i$ with $v_{i+1}\in N_c(v_1,V(P))$ where $2\leq i \leq p-2$, we get that $v_i\notin N_{c^{\prime}}(v_p,V(P))$ since otherwise the cycle $v_1v_2 \cdots v_iv_pv_{p-1}\cdots v_{i+1}v_1$ must be a rainbow cycle where the colours of $v_1v_{i+1}$ and $v_pv_i$ are chosen to be $c$ and $c^{\prime}$. Thus we get
$d_c(v_1,V(P))-1+d_{c^{\prime}}(v_p,V(P))\leq p-2,$
which implies $p \leq p-1$, a contradiction.
\end{proof}

Our first result shows that a rainbow Hamiltonian path exists under a slightly weaker condition than that of Conjecture \ref{rhc}:

\begin{proposition}\label{pro1}
Given graphs $G_1,...,G_n$ on the same vertex set $V$ of size $n$, where $\delta(G_i)\geq \frac{n-1}{2}$ for $i\in [n]$, then there exists a rainbow Hamiltonian path.
\end{proposition}
\begin{proof}
Suppose not, let $P=v_1\cdots v_k$, where $k\leq n-1$, be a rainbow path with the maximum length. Thus there exist at least two colors $c, c^{\prime}$ that are not used by the edges in $P$. Now consider the neighbourhood $N_{c}(v_1)$ and $N_{c^{\prime}}(v_k)$, we have
$$d_{c}(v_1)+d_{c^{\prime}}(v_k)\geq \frac{n-1}{2}+\frac{n-1}{2}=n-1.$$
For each vertex $u\in V-V(P)$, we have $u\notin N_{c}(v_1)\cup N_{c^{\prime}}(v_k)$ and otherwise we can extend $P$ into a larger rainbow path, a contradiction. Thus we get $N_{c}(v_1), N_{c^{\prime}}(v_k)\subseteq V(P)$. However, since $|V(P)|\leq n-1$ and $d_c(v_1,V(P))+d_{c^{\prime}}(v_k,V(P))\geq n-1 \geq |V(P)|$, by Lemma \ref{lemma1} we get a rainbow cycle $C$ of length $k$. Suppose that the colour $c^{\prime\prime}$ is not used by this cycle. Since the monochromatic graph coloured by $c^{\prime\prime}$ is connected, at least one edge $e_0$ coloured by $c^{\prime\prime}$ is between $V(C)$ and $V-V(C)$. Therefore, $V(C)\cup \{e_0\}$ contains a rainbow path with length $k+1$, a contradiction.
\end{proof}

The lower bound here is best possible. One can take $n$ copies of $K_{\frac{n}{2}-1, \frac{n}{2}+1}$ where $n$ is even and there does not exist a rainbow Hamiltonian path since $K_{\frac{n}{2}-1, \frac{n}{2}+1}$ does not contain a Hamiltonian path.

\section{Proof of Theorem \ref{thm1}}
In this section, we give a proof of Theorem \ref{thm1}. Let $G=(V,\bigcup_{i=1}^n E(G_i))$ be the edge-colored multigraph with $G_i$ as the graph of color $i$. We first find a rainbow cycle of length $n-1$ by following a classical proof of Dirac theorem. Then we obtain a rainbow cycle of length $n-2$ or $n-3$ and use it to build cycles of other lengths.

\begin{claim}\label{cla1}
$G$ contains a rainbow cycle of length $n-1$.
\end{claim}
\begin{proof}
By Proposition \ref{pro1}, we first find a rainbow Hamiltonian path $P=v_1v_2\cdots v_{n}$. Without loss of generality, suppose the colour of edge $v_iv_{i+1}$ is $i$ for $1\leq i \leq n-1$ and the only colour that does not appear in $P$ is $n$. Now consider the subpath $P^{\prime}=v_1v_2\cdots v_{n-1}$. Since $|N_{n-1}(v_1)\backslash \{v_n\}| \geq \frac{n-1}{2}$ and $|N_{n}(v_{n-1})\backslash \{v_n\}| \geq \frac{n-1}{2}$, we get $d_{n-1}(v_1,V(P^{\prime}))+d_{n}(v_n,V(P^{\prime}))\geq n-1=|V(P^{\prime})|$. By Lemma \ref{lemma1}, we can find a rainbow cycle of length $n-1$.
\end{proof}

\begin{claim}
$G$ contains either a rainbow cycle of length $n-2$ or a rainbow cycle of length $n-3$.
\end{claim}
\begin{proof}
Suppose that $G$ neither contains a cycle of size $n-3$ nor $n-2$. By Proposition \ref{pro1}, we can find a rainbow path $P_1=v_1v_2\cdots v_{n-3}$ whose order is $n-3$ and, without loss of generality, suppose the colour of edge $v_iv_{i+1}$ is $i$ for $1\leq i \leq n-4$ and the set of colours that are not used in $P_1$ is $S=\{n-3,n-2,n-1,n\}$. We can deduce that $N_{n-1}(v_1)\cap N_{n}(v_{n-3})\cap (V(G)\setminus V(P_1))=\emptyset$ since otherwise we already find a rainbow cycle of length $n-2$, a contradiction. Now we get $d_{n-1}(v_1,V(P_1))+d_{n}(v_{n-3},V(P_1))\geq n-2 \geq |V(P_1)|$. By Lemma \ref{lemma1}, we can find a rainbow cycle of length $n-3$, a contradiction.
\end{proof}

Let $C=v_1\cdots v_{p}$ be a rainbow cycle where $p=n-2$ or $n-3$. In the remainder of the proof, we let $v_i=v_{i-p}$ for $i>p$. We will use the following claim as a tool to analyse the structure of $G$ when it does not contain rainbow cycles of all length $3,...,p+1$.

\begin{claim}\label{lemma2}
Let $c,c^{\prime}$ be two colours not used on $C$ and $x\in V\setminus V(C)$. If $d_c(x,V(C))+d_{c^{\prime}}(x,V(C))\geq p$, then one of the following properties is true:
\begin{enumerate}[(1)]
\item there exist rainbow cycles of length $3,...,p+1$;
\item $d_c(x,V(C))+d_{c^{\prime}}(x,V(C))=p$ and we can partition $V(C)$ into disjoint sets $S_1$ and $S_2$, where $S_1=\{v_{i+j-2}: v_j\in N_c(x,V(C))\}$ and $S_2=N_{c^{\prime}}(x,V(C))$ for some $3\leq i \leq p+1$.
\end{enumerate}
\end{claim}

\begin{proof}
Suppose that $d_c(x,V(C))+d_{c^{\prime}}(x,V(C))\geq p$ and there is no rainbow cycle of length $i$ for some $3\leq i \leq p+1$, thus for each vertex $v_j\in N_c(x,V(C))$ we have $v_{i+j-2}\notin N_{c^{\prime}}(x,V(C))$ since otherwise the cycle $xv_jv_{j+1}\cdots v_{j+i-2}x$ is a rainbow cycle of length $i$ by choosing the colours of $xv_j$ and $xv_{j+i-2}$ to be $c$ and $c^{\prime}$. Therefore, we get $S_1\cap S_2=\emptyset$ by definition. However, since $|S_1|+|S_2|\geq p$ and $S_1\cup S_2 \subseteq V(C)$ it follows that $V(C)$ is partitioned into $S_1$ and $S_2$ and $|S_1|+|S_2|=p$, which implies $d_c(x,V(C))+d_{c^{\prime}}(x,V(C))=p$.
\end{proof}

\begin{case}\label{case1}
$p=n-2$.
\end{case}
Suppose $V(G)\backslash V(C)=\{v^{\prime}, v^{\prime\prime}\}$ and the colours not used by $C$ are $n-1$ and $n$. Suppose that for some $3\leq j \leq n-1$, there does not exist a rainbow cycle of size $j$ in $G$. By Claim \ref{lemma2}, we conclude that $d_{n-1}(v^{\prime}, V(C))+d_{n}(v^{\prime},V(C))\leq n-2$, which implies that $d_{n-1}(v^{\prime},v^{\prime\prime})+d_n(v^{\prime},v^{\prime\prime})\geq n+1-(n-2)=3$.
This is a contradiction. Therefore, $G$ contains rainbow cycles of all sizes $3,...,n-1$.

\begin{case}
$p=n-3$.
\end{case}

Suppose $V(G)\backslash V(C)=\{v_{1}^{\prime},v_{2}^{\prime},v_3^{\prime}\}$ and the colours not used by $C$ are $n-2, n-1$ and $n$. Suppose that for some $3\leq i \leq n-2$, there does not exist a rainbow cycle of size $i$ in $G$. We know that $d_{n-1}(v_3^{\prime}, V(C))+d_{n}(v_3^{\prime},V(C))\geq 2(\frac{n+1}{2}-2)\geq n-3$, thus by Claim \ref{lemma2} we get $d_{n-1}(v_3^{\prime}, V(C))+d_{n}(v_3^{\prime},V(C))=n-3$. This implies that
$d_{n-1}(v_3^{\prime},\{v_{1}^{\prime},v_{2}^{\prime}\})=d_{n}(v_3^{\prime},\{v_{1}^{\prime},v_{2}^{\prime}\})=2$ and hence $\{n-1,n\}\subseteq \text{Col}(v_{3}^{\prime}v_{1}^{\prime})\cap \text{Col}(v_{3}^{\prime}v_{2}^{\prime})$.

By symmetry we now suppose that $\text{Col}(v_{a}^{\prime}v_{b}^{\prime})=\{n-2,n-1,n\}$ for every $1\leq a<b \leq 3$. Let $T_1=\{v_{j+i-3} \mid v_j\in N_{n-2}(v_{1}^{\prime}, V(C))\}$ and $T_2=N_{n-1}(v_{2}^{\prime},V(C))$, by an analogy to the proof of Claim \ref{lemma2}, we find that $T_1\cap T_2=\emptyset$, otherwise suppose that $v_{j+i-3}\in T_2$ for some $j$, we thus have $v_{2}^{\prime}v_{1}^{\prime}v_jv_{j+1}\cdots v_{j+i-3}v_{2}^{\prime}$ is a rainbow cycle with length $i$ by choosing the colours of $v_{2}^{\prime}v_{1}^{\prime}$, $v_{1}^{\prime}v_j$ and $v_{j+i-3}v_{2}^{\prime}$ to be $n$, $n-2$ and $n-1$, which is a contradiction.
We actually get:
$$n-3\leq \frac{n+1}{2}-2+\frac{n+1}{2}-2 \leq d_{n-2}(v_{1}^{\prime},V(C))+d_{n-1}(v_{2}^{\prime},V(C))\leq |V(C)|=n-3,$$
thus all the inequalities above must be equalities and we get $$|T_1|+d_{n-1}(v_{2}^{\prime},V(C))=n-3,$$
which implies that $V(C)$ is partitioned into $T_1$ and $T_2$. Since all colours and vertices are symmetric, the similar conclusion follows by considering $T_1$ and $N_{a}(v_{b}^{\prime},V(C))$ for any $n-1\leq a \leq n$ and $2\leq b \leq 3$. Thus, we finally obtain $T_2=N_{a}(v_{b}^{\prime},V(C))$ for any $n-1\leq a \leq n$ and $2\leq b \leq 3$. Now let $T_1^{\prime}=\{v_{j+i-3} \mid v_j\in N_{n}(v_{1}^{\prime}, V(C))\}$ and recall that $T_2=N_{n-1}(v_{2}^{\prime},V(C))$. Therefore, we reach the similar conclusion that $T_2=N_{a}(v_{b}^{\prime},V(C))$ for any $n-2\leq a \leq n-1$ and $2 \leq b \leq 3$ by considering $T_1^{\prime}$ and $T_2$. This implies $T_2=N_{a}(v_{b}^{\prime},V(C))$ for any $n-2\leq a \leq n$ and $2 \leq b \leq 3$. By symmetry,
we actually get $T_2=N_{a}(v_{b}^{\prime},V(C))$ for any $n-2\leq a \leq n$ and $1\leq b \leq 3$.

Now we claim that there exists some $v_{j_0}\in T_2$ such that $v_{j_0-1}\notin T_2$. Suppose not, for each $v_j\in T_2$ we have $v_{j-1}\in T_2$. This implies $T_2=V(C)$ but this is impossible since $T_1\neq \emptyset$.
Now since $v_{j_0-1}\notin T_2$ we get $v_{j_0-1}\in T_1$ and there exists some $v_{j_1}\in N_{n-2}(v_{1}^{\prime},V(C))$ such that $j_0-1=j_1+i-3$, which implies $j_0=j_1+i-2$. However, in this case the cycle $v_{1}^{\prime}v_{j_1}v_{j_1+1}\cdots v_{j_0}v_{1}^{\prime}$ is a rainbow cycle of length $i$ by choosing the colours of $v_{1}^{\prime}v_{j_1}$ and $v_{1}^{\prime}v_{j_0}$ to be $n-2$ and $n-1$, a contradiction. This concludes the proof.
\qed

\section{Proof of Theorem \ref{main}}

In this section, we give a proof of Theorem \ref{main} by proving a rainbow type of absorbing lemma. By $0<\alpha \ll \beta$ we mean that there exists an increasing function $f: \mathbb{R}\rightarrow \mathbb{R}$ such that the subsequent argument is valid for any $0<\alpha \leq f(\beta)$. We first introduce the absorbing lemma for Theorem \ref{main}:

\begin{lemma}\label{lm1}
 Let $n, \mu, \varepsilon$ be such that $\frac{1}{n}\ll \mu \ll \varepsilon$. For any graphs $G_1,...,G_n$ on the same vertex set of size $n$, each graph having minimum degree at least $(\frac{1}{2}+\varepsilon)n$, there exists a rainbow cycle $C$ with length at most $\mu n$ such that for every rainbow path $P$ with $V(P)\cap V(C)=\emptyset$ and $\text{Col}(P)\cap \text{Col}(C)=\emptyset$, if $s$ is a colour that is not used by $C$ and $P$, then there exists a rainbow cycle $C^{\prime}$ with
\begin{enumerate}[(i)]
  \item $V(C^{\prime})=V(C)\cup V(P)$;
  \item $\text{Col}(P \cup C)\cup \{s\} $=$\text{Col}(C^{\prime})$.
\end{enumerate}
\end{lemma}

Combining Lemma \ref{lm1} and Proposition \ref{pro1}, we immediately reach a proof of Theorem \ref{main} as follows:

\begin{proof}[Proof of Theorem \ref{main}]
Let $C$ be the absorbing cycle we given in Lemma \ref{lm1}. Now for each $i$, let $G^{\prime}_i = G_i-V(C)$, the subgraph of $G_i$ induced on $V(G_i)\setminus V(C)$. Then $\delta(G^{\prime}_i) \geq \frac{1+\varepsilon}{2}|V(G^{\prime}_i)|$. Let $W$ be the set of colours that do not appear on any edge of $C$. One can thus construct a rainbow Hamiltonian path $P_1=v_0v_1\cdots v_t$ using exactly $|W|-1$ colours of $W$ by Proposition \ref{pro1}. Suppose $s_1$ is the unique colour in $W$ that is not used by $P_1$. By Lemma \ref{lm1} there exists a rainbow cycle $C^{\prime}$ with $V(C^{\prime})=V(C)\cup V(P)$ and $\text{Col}(P \cup C)\cup \{s_1\}=\text{Col}(C^{\prime})$, which implies $C^{\prime}$ is a rainbow Hamiltonian cycle and thus we conclude the proof.
\end{proof}

In the remaining part of this section, we give a proof of Lemma \ref{lm1}. First we introduce some notation and basic results.
For any pair of two not necessarily distinct vertices $x_1,x_2\in V(G)$ and four distinct colours $1\leq s,i,j,k \leq n$, we define the set of absorbing paths $A_{s,i,j,k}(x_1,x_2)$ to be the family of edge-coloured  $3$-paths $P$ that satisfy the following conditions:

(i) $P=v_1v_2v_3v_4$ where $\{v_1,v_2,v_3,v_4\}\cap \{x_1,x_2\}=\emptyset$;

(ii) the edges $v_1v_2,v_2v_3,v_3v_4$ are coloured respectively by $i,j,k$;

(iii) $s\in \text{Col}(x_1v_2)$ and $j\in \text{Col}(x_2v_3)$.

For every path $P$ in $A_{s,i,j,k}(x_1,x_2)$, we say that $P$ is an \emph{absorbing path} for $(x_1,x_2)$ with colour pattern $(s,i,j,k)$. In practice, we always choose $x_1$ and $x_2$ to be two endpoints of some path $Q$.

\begin{claim}\label{cl1}
For each pair $(x_1,x_2)$ and four distinct colours $s,i,j,k$, we have $|A_{s,i,j,k}(x_1,x_2)|$ $\geq \frac{\varepsilon n^4}{8}$ when $n$ is sufficiently large.
\end{claim}

\begin{proof}
First, choose a vertex $v_2\in N_s(x_1)\backslash\{x_2\}$. Pick another vertex $v_3\in (N_j(v_2)\cap N_j(x_2))$ $\backslash \{x_1\}$. The total number of such $v_2,v_3$ is at least $(\frac{n}{2}+\varepsilon n-1)(2\varepsilon n-1)$. Now fix $v_2$ and $v_3$. Choose $v_1\in N_i(v_2)\backslash \{x_1,x_2,v_3\}$. Choose another vertex $v_4\in N_k(v_3) \backslash \{x_1,x_2,v_1,v_2\}$. Note that the total number of such $v_1,v_4$ is at least $(\frac{n}{2}+\varepsilon n-3)(\frac{n}{2}+\varepsilon n-4)$ and hence we derive that there exist at least
$$\left(\frac{n}{2}+\varepsilon n-1\right) \left(2\varepsilon n-1 \right) \left(\frac{n}{2}+\varepsilon n-3\right)\left(\frac{n}{2}+\varepsilon n-4\right)\geq \frac{\varepsilon n^4}{8}$$
absorbing paths for $(x_1,x_2)$ when $n$ is sufficiently large.
\end{proof}

\begin{proof}[Proof of Lemma \ref{lm1}]
Let $\mu_1=\mu/5$ and $\ell$ be new constant such that $\lceil\mu_1 n\rceil-1 \leq \ell\leq \lceil\mu_1 n\rceil+1$ and $\ell$ is divisible by $3$. For simplicity, we assume that $\ell=\mu_1 n$. We fix $\ell/3$ groups of colours $C_i=\{3i-2,3i-1,3i\}$ where $i=1,..., \ell/3$. Let $\mathcal{P}_{C_i}$ be the set of all the paths $P=v_0v_1v_2v_3$ in $G$ where the colours of $v_0v_1,v_1v_2, v_2v_3$ are $3i-2, 3i-1, 3i$ for all $1\leq i \leq \ell/3$.

Now consider a random set $W$ by selecting an element from each $\mathcal{P}_{C_i}(i \in [\ell/3])$ independently where every element in $\mathcal{P}_{C_i}$ is chosen with probability $1/|\mathcal{P}_{C_i}|$. For any colour $s$ and any pair $(x_1,x_2)$, set $A_s(x_1,x_2)=\bigcup_{i=1}^{\ell/3}(A_{s,3i-2,3i-1,3i}(x_1,x_2)\cap W)$.
Now for each $i \in [\ell/3]$, let the random variable $X_i$ be the indicative variable of the event that $W \cap A_{s,3i-2,3i-1,3i}(x_1,x_2)\neq \emptyset$. Hence we get $|A_s(x_1,x_2)|=\sum_{i=1}^{\ell/3}X_i$ and all $X_i$'s are independent. Using Claim \ref{cl1}, we get
$$P(X_i=1)=|A_{s,3i-2,3i-1,3i}(x_1,x_2)|/ |\mathcal{P}_{C_i}| \geq \frac{\varepsilon n^4}{8}/ n^4\geq \frac{\varepsilon}{8}$$
for $i \in [\ell/3]$ and hence
$$E(|A_s(x_1,x_2)|)=\sum_{i=1}^{\ell/3}E(X_i)\geq \frac{\varepsilon \ell}{24}.$$
By Lemma \ref{chernoff} with $\varepsilon=1/2$, we see that
$$P\left(|A_s(x_1,x_2)|<\frac{\varepsilon \ell}{48}\right)\leq 2e^{-\frac{\varepsilon \ell}{288}}\leq 2e^{-\frac{\varepsilon \mu_1 n}{10^3}}.$$
Now let $Y$ be the number of pairs of 3-paths in $W$ that intersect with each other. For some distinct $1\leq i,j \leq \ell/3$, let $Y_{i,j}$ be the indicative variable of the event that the path we choose in $A_{C_i}$ intersects with the path we choose in $A_{C_j}$. Thus we have $Y=\sum_{i,j}Y_{i,j}$. We claim that the size of set $\{\{P_1,P_2\}\mid P_1\in A_{C_i}, P_2\in A_{C_j}, P_1, P_2$ are intersecting with each other$\}$ is at most $16n^7$ for fixed $i,j$. Since the number of $P_1$ is at most $n^4$, and when $P_1$ is fixed, the number of $P_2$ that we can choose is at most $16n^3$. Besides, it is obvious that when $P_1,P_2$ are fixed, the probability that we have chosen $P_1,P_2$ together is $\frac{1}{|A(C_i)||A(C_j)|} \leq \frac{1}{(n^4/8)^2}$ because $|A_{C_i}|,|A_{C_j}| \geq \frac{n^4}{8}$ when $n$ is sufficiently large.
Therefore, we get
$$E(Y) \leq \binom{\ell/3}{2}\cdot16\cdot n^7 \cdot \frac{1}{(n^4/8)^2}\leq 10^3\mu_1^2 n\leq \frac{\varepsilon \mu_1 n}{200}.$$
Using Markov's inequality, we get that
 $$P\left(Y\geq \frac{\varepsilon \mu_1 n}{100}\right)\leq \frac{1}{2}.$$
Now choose sufficiently large $n$ such that
$$2n^3e^{-\frac{\varepsilon \mu_1 n}{10^3}}+\frac{1}{2}<1.$$
Thus by the union bound, with positive possibility, for each $s$ and any pair $(x_1,x_2)$ we have
(i) $|A_s(x_1,x_2)|\geq \frac{\varepsilon \ell}{48} \geq \frac{\varepsilon \mu_1 n}{48}$, and
(ii) $Y< \frac{\varepsilon \mu_1 n}{100}$.

Fix such $W$, we delete one 3-paths in each intersecting pair of $W$. Suppose that the remaining path family is $W^{\prime}$. Thus $W^{\prime}$ is a family containing mutually disjoint 3-paths and for every $s$ and any pair $(x_1,x_2)$ we get that
$$\left|\bigcup_{i=1}^{\ell/3}(A_{s,3i-2,3i-1,3i}(x_1,x_2)\cap W^{\prime})\right|\geq \frac{\varepsilon \mu_1 n}{48}-\frac{\varepsilon \mu_1 n}{100}\geq \frac{\varepsilon \mu_1 n}{100}.$$

Let $W^{\prime}=\{P_1,...,P_t\}$ be the path family we found before and let $S$ be the set of the colours that do not appear in any path in $W^{\prime}$. Let $V^{\prime}=V(G)\setminus \bigcup_{i=1}^tV(P_i)$. Without loss of generality, we suppose that $P_i=v_1^{(i)}v_2^{(i)}v_3^{(i)}v_4^{(i)}$ for $1\leq i \leq t$. Now for $P_1,P_2$, it is obvious that we can find a vertex $u_1\in V^{\prime}$ such that $u_1v_4^{(1)}$ and $u_1v_1^{(2)}$ are two edges coloured with distinct colours in $S$. Delete these two colours from $S$ and the vertex $u_1$ from $V^{\prime}$. Repeat the above process for the path pair $\{P_2,P_3\},...,\{P_t,P_1\}$, and at last we find $u_1,...,u_t$ and a rainbow cycle $C$ with size at most $5\ell=\mu n$ that contains all the vertices in $\bigcup_{i=1}^t V(P_i)$ and those $u_i$ where $1\leq i \leq t$. For every rainbow path $P \subseteq V(G)-V(C)$ such that the colour set of $P$ is disjoint with the colour set of $C$, if $x_1,x_2$ are two endpoints of $P$ and $s$ is a colour that does not appear in $C$ and $P$, then the pair $(x_1,x_2)$ has at least one absorbing path $P_0=u_1u_2u_3u_4$ in $C$ with colour pattern $(s,3i-1,3i-2,3i)$ for some $i \in [\ell/3]$ since $\frac{\varepsilon \mu_1 n}{100} \geq 1$ when $n$ is sufficiently large. Therefore, we insert the path $P$ into the cycle $C$ to get a rainbow cycle $\{C-u_2u_3\}\cup u_2Pu_3$ where $x_1u_2$ is coloured by $s$ and $x_2u_3$ is coloured by $3i-2$, which completes our proof.
\end{proof}

\section{Acknowledgements}
The first two authors are supported by the National Natural Science Foundation of China (11631014, 11871311) and Shandong University multidisciplinary research and innovation team of young scholars. The third author is partially supported by NSF grant DMS 1700622.
The authors thank referees for thorough comments that improve the presentation of this paper.

\bibliographystyle{plain}
\bibliography{Bibte}

\end{document}